\documentclass[reqno,10pt]{amsart}
\usepackage{amscd,amssymb,amsmath,color}
\usepackage{latexsym}
\usepackage[all]{xy}
\usepackage{defs}
\usepackage[colorlinks=true]{hyperref}

\def\..{{,\dots,}}

\begin{document}

\author{Michael Temkin}
\title{Altered local uniformization of Berkovich spaces}

\address{Einstein Institute of Mathematics, The Hebrew University of Jerusalem, Giv'at Ram, Jerusalem, 91904, Israel}
\email{temkin@math.huji.ac.il}
\thanks{This work was supported by the Israel Science Foundation (grant No. 1018/11).}
\begin{abstract}
We prove that for any compact quasi-smooth strictly $k$-analytic space $X$ there exist a finite extension $l/k$ and a quasi-\'etale covering $X'\to X\otimes_kl$ such that $X'$ possesses a strictly semistable formal model. This extends a theorem of U. Hartl to the case of the ground field with a non-discrete valuation.
\end{abstract}

\maketitle

\section{Introduction}

\subsection{Motivation}
In \cite[Theorem~1.4]{Hartl}, U. Hartl proved that if $X$ is a quasi-compact and quasi-separated smooth rigid space over a complete discretely-valued field $k$ then there exist a finite extension $l/k$ and an \'etale covering $X'\to X\otimes_kl$ such that $X'$ is affinoid and possesses a strictly semistable formal model. Hartl's proof involves an advanced rigid-analytic technique, and as a drawback one has to impose the assumption that $k$ is discretely-valued. Nevertheless, already in this form the theorem had applications, for example, to representability of the rigid-analytic Picard functor.

Recently, the author incorrectly used Hartl's result in \cite{Temkintopforms} to control the maximality locus of pluricanonical forms on quasi-smooth Berkovich analytic spaces over a ground field $k$ of residue characteristic zero. If $k$ is discretely-valued then one even has the semistable reduction theorem available, so the whole point was to deal with the non-discrete case, and I am grateful to W. Gubler for pointing out the limiting hypothesis I missed.

The aim of this paper is to extend Hartl's theorem to arbitrary complete real-valued ground fields, see Theorem~\ref{berth}. In particular, this suffices for applications in \cite{Temkintopforms}. In addition, our method is easier than the original proof.

\begin{rem}
(i) The author conjectures that one can even achieve that $X'$ is a disjoint union of affinoid domains in $X\otimes_kl$ and calls this the local uniformization conjecture for Berkovich spaces using an analogy with the classical local uniformization conjecture for varieties. In particular, it is natural to call Theorem~\ref{berth} altered (or weak) local uniformization of Berkovich spaces.

(ii) Both our strengthening of Hartl's theorem and the local uniformization conjecture over a non-discretely valued field $k$ may look surprising because analogous global conjectures only predict existence of logarithmically smooth formal models rather than semistable ones. So, let us illustrate the situation with a simple example. If non-zero elements $\pi,\omega\in\kcirc$ are such that $\log|\pi|$ and $\log|\omega|$ are linearly independent then the product $X=\calM(k\{t,\pi t^{-1},s,\omega s^{-1}\})$ of annuli of radii $|\pi|$ and $|\omega|$ has no semistable formal model. This can be easily deduced from the fact that the skeleton $\Delta=[\log|\pi|,0]\times[\log|\omega|,0]$ of $X$ cannot be triangulated using triangles with rational slopes. (A similar example of Karu, \cite[Section~0.4]{Abramovic-Karu}, shows that in the weak semistable modification theorem of Abramovich-Karu, \cite[Theorem~0.3]{Abramovic-Karu}, one has to weaken the usual notion of semistability.) Nevertheless, if we allow overlaps then $\Delta$ can be easily covered by such triangles, and this allows to construct a covering of $X$ by affinoid domains with semistable reduction.
\end{rem}

\subsection{The method}
The problem immediately reduces to the following formal version: given a rig-smooth formal scheme $\gtX$ over the ring of integers $\kcirc$ find a finite extension $l/k$ and a rig-\'etale covering $\gtX'\to\gtX\otimes_{\kcirc}\lcirc$ such that $\gtX'$ is strictly semistable. Furthermore, $\gtX$ is locally algebraizable by Elkik's theorem, hence the problem reduces to the following algebraic version: if $S=\Spec(\kcirc)$ and $X$ is a flat generically smooth $S$-scheme of finite type then there exists a finite extension $l/k$ and a covering $X'\to X\otimes_{\kcirc}\lcirc$ with a strictly semistable $X'$. This time we work with the topology of certain generically \'etale coverings, see Section~\ref{topsec}.

The algebraic version of the theorem is proved by fibering the scheme $X$ in curves and inducting on the dimension. This approach is inherent for de Jong's proof of the alteration theorem and its numerous successors, including Hartl's theorem and various desingularization results of Gabber, such as the $l$'-alteration theorem in \cite{X}. To run the induction step one has to somehow resolve a relative curve $X\to Y$. In all mentioned results, one uses de Jong's approach via the moduli spaces of $n$-pointed stable curves, although one has to by-pass some troubles due to the need to compactify the curve. For example, Hartl mentions in the end of the introduction to \cite{Hartl} that his method has to use rigid-analytic technique because otherwise the compactification would not be possible.

The main novelty of our method is in the use of the stable modification theorem of \cite{temst} and \cite{temrz} instead of the moduli space approach. This theorem applies to arbitrary, even non-separated, relative curves and produces a canonical semistable modification $X'\to X\times_YY'$, where $g\:Y'\to Y$ is a sufficiently large covering for a suitable topology ($U$-\'etale coverings in the sense of \cite[\S2.3]{temrz}). In addition, it provides a control on the base change morphism $g$. Namely, if $X\to Y$ is smooth over $U\subseteq Y$ then one can choose $g$ to be \'etale over $U$. It is the latter refinement that allows us to ensure that the morphism $X'\to X$ is \'etale over the generic point of $S$. The remaining argument is pretty standard so we do not discuss it here.

\subsection{Structure of the paper}
Section \ref{algsec} is devoted to the algebraic version of our main result, which is proved in Theorem~\ref{mainth}. The formal and non-archimedean versions are deduced in Section~\ref{unifsec}, see Theorems~\ref{formalth} and \ref{berth}.

\subsection{Future research}\label{future}
The goal of this paper is to prove Theorem \ref{berth} in the most economical way, so we postpone a more systematic study to another paper. Here we only note that the method can be strengthened similarly to Gabber's $l$'-alteration theorem. Moreover, if the residue characteristic is zero, one can even establish the actual local uniformization of $X$ in this way.

\subsection{Acknowledgments}
I am grateful to the anonymous referee for pointing out gaps and inaccuracies in the first version of the paper.

\section{The algebraic case}\label{algsec}

\subsection{Conventiones}

\subsubsection{The valued field}
Throughout Section \ref{algsec}, $k$ denotes a valued field of height one. By $|\ |\:k\to\bfR_{\ge 0}$, $\kcirc$, $\kcirccirc$, and $\tilk=\kcirc/\kcirccirc$ we denote the valuation, the valuation ring of $k$, the maximal ideal of $\kcirc$, and the residue field of $k$, respectively.

\subsubsection{The base scheme}\label{Ssec}
The scheme $S=\Spec(\kcirc)$ consists of two points: $\eta=\Spec(k)$ and $s=\Spec(\tilk)$. We have the generic fiber functor $X\mapsto X_\eta=X\times_S\eta$ from $S$-schemes to $k$-schemes. By $X_s$ we denote the closed fiber of $X$. An $S$-morphism $f$ is called {\em $\eta$-\'etale}, {\em $\eta$-isomorphism}, etc., if its generic fiber $f_\eta$ is so. We say that $f$ is an {\em $\eta$-modification} if it is a proper {\em $\eta$-isomorphism}. A typical example of an $\eta$-modification is an {\em admissible blow-up}, i.e. a blow-up whose center is disjoint from $X_\eta$.

\subsection{The $\eta$-\'etale $S$-admissible topology}

\subsubsection{Flat $S$-schemes}
For an $S$-scheme $X$ the following conditions are equivalent: (i) $X$ is $S$-flat, (ii) $\calO_X$ has no $\kcirc$-torsion, (iii) $X_\eta$ is schematically dense in $X$. To any $S$-scheme $X$ one can associate the flat $S$-scheme $X^\st$ which is the schematic closure of $X_\eta$ in $X$ (it can be viewed as an analogue of strict transform.) Since the functor $X\mapsto X^\st$ is left adjoint to the embedding of the category of flat $S$-schemes into the category of $S$-schemes, fibered product exist in the category of flat $S$-schemes and are given by the formula $(Y\times_XZ)^\st$. We say that $(Y\times_XZ)^\st\to Z$ is the {\em strict base change} of $Y\to X$.

\subsubsection{Admissible $S$-schemes}
We say that an $S$-scheme $X$ is {\em admissible} if it is flat and of finite type over $S$. Note that in this case $X$ is automatically of finite presentation over $S$ by \cite[Part I, Corollary~3.4.7]{RG}. By an {\em admissible valuation} of an $S$-admissible scheme $X$ we mean a morphism $\alp\:\Spec(R)\to X$, where $R$ is a valuation ring and $\alp$ takes the generic point to $X_\eta$.

\subsubsection{Topology $\tau$}\label{topsec}
We provide the category of admissible $S$-schemes with the following topology $\tau$ that we call the {\em $\eta$-\'etale $S$-admissible topology}: an $S$-morphism $f\:Y\to X$ is a $\tau$-covering if it is $\eta$-\'etale and satisfies the following lifting property: (*) any admissible valuation $\alp\:\Spec(R)\to X$ with an algebraically closed $\Frac(R)$ factors through (or lifts to) an admissible valuation $\beta\:\Spec(R)\to Y$. It is easy to see that $\tau$-coverings are preserved by compositions and strict base changes and hence, indeed, form a Grothendieck topology. In fact, this topology was already considered in \cite{temrz}:

\begin{rem}
Recall that $f\:Y\to X$ is an {\em $X_\eta$-\'etale covering} in the sense of \cite[\S2.3]{temrz} if it is $\eta$-\'etale and for any admissible valuation $\alp\:\Spec(R)\to X$ there exists a domination of valuation rings $R\subseteq R'$ such that $\Frac(R')/\Frac(R)$ is finite and the composed valuation $\alp'\:\Spec(R')\to X$ factors through $Y$. Plainly, any $X_\eta$-\'etale covering is a $\tau$-covering. The converse is also true and easy but will not be used, so we leave it to the interested reader.
\end{rem}

\subsubsection{An alternative description of $\tau$}
By the valuative criterion of properness, any $\eta$-modification $X'\to X$ is a $\tau$-covering. Also, it is easy to see that any flat surjective $\eta$-\'etale morphism $f\:Y\to X$ is a $\tau$-covering. Indeed, assume that $\alp\:T=\Spec(R)\to X$ is an admissible valuation and $K=\Frac(R)$ is algebraically closed. Being flat and $\eta$-\'etale $f$ is quasi-finite, and hence the base change $g\:T'=T\times_XY\to T$ is a flat quasi-finite covering. It suffices to show that $g$ has a section $s\:T\to T'$ since then the composition $T\to T'\to Y$ provides a lifting of $\alp$. To find $s$ we can replace $T'$ by its reduction. Next, choose a point $t'$ over the closed point of $T$. Then $\calO_{T',t'}$ is a reduced local ring dominating $R$. Since $K$ is algebraically closed, $\Frac(\calO_{T',t'})=K$ and hence $\calO_{T',t'}=R$, giving rise to the section $T=\Spec(\calO_{T',t})\into T'$.

We claim that, conversely, $\tau$ is generated by these two types of coverings. In fact, we even have the following more precise result.

\begin{lem}\label{coverlem}
Assume that $f\:Y\to X$ is a $\tau$-covering. Then there exists an $\eta$-modification $Y'\to Y$ such that $Y'\to X$ factors as a composition of a flat surjective $\eta$-\'etale morphism $Y'\to X'$ and an admissible blow-up $X'\to X$.
\end{lem}
\begin{proof}
By the flattening theorem of Raynaud and Gruson, see \cite[Part I, Th\'eor\`em~5.7.9]{RG}, there exists an admissible blow-up $g\:X'=\Bl_\calI(X)\to X$ such that the strict base change $Y'$ of $Y$ is flat over $X'$. The morphism $f'\:Y'\to X'$ is flat and has the same generic fiber as $f$. In addition, the morphism $Y'\to Y$ is clearly an $\eta$-modification (in fact, it is easy to see that it is even the admissible blow-up along $f^{-1}\calI$). So, it remains to show that $f'$ is surjective.

Fix a point $x'\in X'$. Since $X'$ is admissible it is easy to see that there exists an admissible valuation $\alp'\:T=\Spec(R)\to X'$ sending the closed point $t\in T$ to $x'$. Clearly, we can also achieve that $\Frac(R)$ is algebraically closed. Let $\alp$ be the composition $T\to X'\to X$. We claim that any lifting $\alp''\:T\to X'$ of $\alp$ coincides with $\alp'$. Indeed, by the properness of $g$ it suffices to check that $\alp''(\veps)=\alp'(\veps)$ for the generic point $\veps\in T$, but this is clear since $\alp'(\veps)\in X'_\eta$ and $X'_\eta=X_\eta$.

Since $f$ is a $\tau$-covering, $\alp$ lifts to a valuation $\beta\:T\to Y$, which further lifts to $\beta'\:T\to Y'$ by the properness of $Y'\to Y$. We showed that $\alp'$ is the only lifting of $\alp$ and hence the lifting $\beta'$ of $\alp$ is also a lifting of $\alp'$. Thus $f'(\beta'(t))=\alp'(t)=x'$ and so the fiber $f'^{-1}(x')$ is non-empty.
\end{proof}

\begin{rem}
(i) We chose to work with the topology $\tau$, but there are a few more nearly equivalent choices that would fit our aims equally well. It seems that none of them was developed enough in the literature in the non-noetherian setting, so we chose the one that required minimal preparation. For example, an alternative choice was to use the topology of alterations in the sense of \cite{II} with the restriction that all coverings are $\eta$-\'etale. It is known to experts that the topology of alterations works fine in the non-noetherian setting when there are finitely many generic points (which is the case for admissible $S$-schemes), see \cite[Remark~1.2.4(ii)]{II}.

(ii) There are various refinements of $\tau$. For example, one can restrict the degrees of the extensions $[k(y):k(f(y))]$ for generic points $y\in Y$. Some of our results hold for finer topologies, but we will not explore this direction in the paper.
\end{rem}

\subsection{Local structure of semistable $S$-schemes}

\subsubsection{Nodal curves}\label{nodalsec}
Let $f\:Y\to X$ be a finitely presented morphism. Recall that $f$ is called a {\em nodal} or {\em semistable} curve if it is flat and all geometric fibers are nodal curves in the sense that any singularity of such a fiber (if exists) is an ordinary double point. It is well known that $f$ is a nodal curve if and only if \'etale-locally it is isomorphic to a morphism of the form $\Spec(A[x,y]/(xy-\pi))\to\Spec(A)$ where $\pi\in A$. For example, by the approximation technique of \cite[$\rm IV_3$, Section~8]{ega} this immediately reduces to the case when $X$ is of finite type over $\bfZ$ and then one can deduce this from \cite[2.23]{dJ} since working \'etale-locally one can split both the curve and the quadratic form $Q$ in loc.cit.

\subsubsection{Semistable $S$-schemes}
A higher-dimensional semistability can be defined for all morphisms but we restrict to $S$-schemes. Recall that an $S$-scheme $X$ is called {\em strictly semistable} (resp. {\em semistable}) if locally (resp. \'etale-locally) it admits an \'etale morphism to a {\em model semistable $S$-scheme}
\begin{equation}\label{eq}
\Spec(\kcirc[t_0\..t_l]/(t_0\dots t_m-\pi)),
\end{equation}
where $0\neq\pi\in\kcirc$ and $0\le m\le l$.

\begin{rem}
The condition $\pi\neq 0$ is imposed in order to guarantee that the generic fiber of $X$ is smooth. There is an alternative definition where $\pi\in\kcirc$ is arbitrary, and then $X_\eta$ can be a normal crossings variety over $k$.
\end{rem}

\subsubsection{A local description}
For a model semistable $S$-scheme $Z$ as in (\ref{eq}) let $O$ denote the point where $t_0\..t_l$ vanish. We call $O$ the {\em origin} of $Z$. The following result is well-known but hard to find in the literature, so we indicate a proof.

\begin{lem}\label{loclem}
Assume that $k$ is algebraically closed. Then for any strictly semi\-stable $S$-scheme $X$ with a closed point $x\in X_s$ there exists an \'etale morphism $f\:U\to Z$ such that $U$ is a neighborhood of $x$, $Z$ is a model semistable $S$-scheme and $f(x)$ is the origin of $Z$.
\end{lem}
\begin{proof}
Since $X_s$ is a variety over the algebraically closed field $\tilk$, a point $x\in X_s$ is closed if and only if $k(x)=\tilk$. Since $X$ locally admits an \'etale morphism $f$ to a model scheme $Z$ as in (\ref{eq}), it suffices to prove the claim for $Z$ and the closed point $z=f(x)\in Z_s$. Let $n\ge 1$ be the number of zeros among $t_0(z)\..t_m(z)$. Renumbering $t_0\..t_m$ we can assume that $t_i(z)=0$ for $0\le i\le n-1$. Set $t'_i=t_i$ for $0\le i\le n-2$ and $t'_{n-1}=t_{n-1}\dots t_m$. For any $n\le i\le m$ lift $t_i(z)\in k(z)=\tilk$ to an element $c_i\in\kcirc$ and set $t'_i=t_i-c_i$. Then we obtain a morphism $g\:Z\to\Spec(\kcirc[t'_0\..t'_m]/(t'_0\dots t'_{n-1}-\pi))$ that takes $z$ to the origin, and one can easily check that $g$ is \'etale at $z$.
\end{proof}

\subsubsection{Generic units of $\calO_{Z,O}$}
Let $A=\kcirc[t_0\..t_l]/(t_0\dots t_m-\pi)$ and $R=A_q$ the localization at the prime ideal $q=(\kcirccirc,t_0\..t_l)$. So, $Z=\Spec(A)$ is a model semistable $S$-scheme and $R=\calO_{Z,O}$. Our next aim is to describe $R\cap(R_\eta)^\times$, where $R_\eta=R\otimes_{\kcirc}k$, and for this we will use certain monomial semivaluations on $R$. In fact, they are related to skeletons of Berkovich spaces, but we will give a simple ad hoc definition.

Let $\Delta$ be the set of tuples $r=(r_0\..r_l)\in[0,1]^{l+1}$ such that $r_0\dots r_m=|\pi|$. For any $r\in\Delta$ the rule $|a_nt^n|_r=|a_n|r^n=|a_n|r_0^{n_0}\dots r_l^{n_l}$ defines a character on the set of monomials $a_nt^n$, where $n\in\bfN^{l+1}$ and $a_n\in \kcirc$. Next, let us naturally extend $|\ |_r$ to a semivaluation on $R$. Note that replacing $t_0$ by $\pi t_1^{-1}\dots t_m^{-1}$ one can uniquely represent any $a\in A$ as a Laurent polynomial $\sum_{n\in\bfZ^m\times\bfN^{l-m}} a_nt^n$ such that $|a_n|\le|\pi|^{-\min(0,n_1\..n_m)}$. We call this the {\em special representation} of $a$ and the elements $a_nt^n$ as above will be called {\em special monomials}. Using the special representation of $a$ the formula $|a|_r=\max_{n\in\bfZ^m\times\bfN^{l-m}}|a_n|r^n$ defines a (multiplicative) semivaluation $|\ |_r\:A\to\bfR_{\ge 0}$. Furthermore, being generated by $t_i$ with $r_i=0$, the kernel of $|\ |_r$ is contained in $q$, and hence $|\ |_r$ extends by multiplicativity to a semivaluation $|\ |_r\:R\to\bfR_{\ge 0}$. Thus, we can view $\Delta$ as a set of $\kcirc$-semivaluations on $R$ and then any $b\in R$ defines a function $|b|_\Delta\:\Delta\to\bfR_{\ge 0}$ by sending $r$ to $|b|_r$.

\begin{lem}\label{newlem}
Assume that $a=a_nt^n$ and $a'=a'_{n'}t^{n'}$ are two special monomials. Then the following conditions are equivalent:

(i) $a'\in aA$,

(ii) $|a_n|\ge |a'_{n'}|$, $|a_n\pi^{n_i}|\ge|a'_{n'}\pi^{n'_i}|$ for $1\le i\le m$, and $n_i\le n'_i$ for $i>m$.

(iii) $a$ {\em dominates} $a'$ in the sense that $|a|_\Delta\ge|a'|_\Delta$ (i.e. $|a|_r\ge |a'|_r$ for any $r\in\Delta$).
\end{lem}
\begin{proof}
(i)$\Longleftrightarrow$(ii) For any $b\in A$ the special representations of $b$ and $ab$ contain the same number of non-zero special monomials, in particular, $b$ is a special monomial if and only if $ab$ is. Thus, $a'\in aA$ if and only if there exists a special monomial $b=b_ut^u$ such that $a'=ab$. Clearly, this happens if and only if $n_i\le n'_i$ for $i>m$ and the monomial $b=(a'_{n'}/a_n)t^{n'-n}$ is special, that is, $$|a'_{n'}/a_n|\le|\pi|^{\max(0,n_1-n'_1\..n_m-n'_m)}.$$ The latter means that $|a'_{n'}/a_n|\le 1$ and $|a'_{n'}/a_n|\le\pi^{n_i-n'_i}$ for $1\le i\le m$.

(ii)$\Longleftrightarrow$(iii) In (iii), one compares the power functions $|a|_\Delta=|a_n|r^n$ and $|a'|_\Delta=|a'_{n'}|r^{n'}$ on $\Delta$. Note that $\Delta=\Delta_m\times[0,1]^{l-m}$, where $\Delta_m\subset[0,1]^{m+1}$ is an (exponential) $m$-dimensional simplex with vertices $v_0\..v_m$ of the form $(1\.. 1,|\pi|,1\..1)$. For $0\le j\le m$ set $\rho_j=(v_j,1\dots 1)\in\Delta$, then it is easy to see that $|a_n|r^n\ge |a'_{n'}|r^{n'}$ on $\Delta$ if and only if $n_i\le n'_i$ for $i>m$ and $|a_n|\rho_j^n\ge|a'_{n'}|\rho_j^{n'}$ for $0\le j\le m$. For $j=0$ this gives $|a_n|\ge|a'_{n'}|$ and for $1\le j\le m$ this gives $|a_n\pi^{n_j}|\ge|a'_{n'}\pi^{n'_j}|$.
\end{proof}

If $a=\sum_{n\in\bfZ^m\times\bfN^{l-m}} a_nt^n$ and $a_dt^d$ dominates all special monomials in this representation then we call $a_dt^d$ the {\em dominating monomial}.

\begin{lem}\label{locallem}
Keep the above notation, and let $a=\sum_{n\in\bfZ^m\times\bfN^{l-m}} a_nt^n$ be the special representation of a non-zero element of $A$.

(i) The following conditions are equivalent: (a) $a\in R^\times$, (b) $|a_0|=1$, (c) $|a|_\Delta$ is the constant function 1.

(ii) The following conditions are equivalent: (d) $a\in R^\times_\eta$, (e) there exists a dominating monomial $a_dt^d$ and $d\in\bfZ^m\times\{0\}\subset\bfZ^m\times\bfN^{l-m}$, (f) $a=ua_dt^d$, where $u\in A\cap R^\times$ and $d\in\bfZ^m\times\{0\}$.
\end{lem}
\begin{proof}

(i) Note first that $|a|_\Delta$ is the maximum of the power functions $|a_nt^n|_\Delta=|a_n|r^n$. Since $|a_n|\le 1$ for any $n$, it follows that (b)$\Longleftrightarrow$(c).

Let us show that (a)$\Longleftrightarrow$(b). Since $a\in R^\times$ if and only if $a\notin q$, it suffices to prove that any special monomial $a_nt^n$ with $n\neq 0$ lies in $q$. The latter is easily checked by use of Lemma~\ref{newlem}: if $n_i>\min(0,n_1\..n_m)$ for $1\le i\le m$ then $t_i|a_nt_n$, otherwise $0>n_1=\dots=n_m$ and then $a_nt^n$ is divisible by $t_0=\pi t_1^{-1}\dots t_m^{-1}$.

(ii) Recall that $|a|_\Delta$ is the maximum of power functions on $\Delta$. The crucial observation is that a product of two functions of this form is a power function on $\Delta$ if and only if both of them are power functions on $\Delta$. Perhaps the easiest way to show this is by linearizing the question via logarithms: $\log(|a|_\Delta)$ is the maximum of linear functions of $\log(r_i)$, and by convexity considerations if a sum of two such functions is linear then both summands are linear.

Thus, if $a,a'\in A$ are such that $|a|_\Delta|a'|_\Delta$ is a non-zero constant on $\Delta$, then $|a|_\Delta$ is a power function on $\Delta$. Clearly, $|a|_\Delta$ does not attain zero in this case, and hence $|a|_\Delta=cr^d$ for some $d\in\bfZ^m\times\{0\}$. It follows that $|a|_\Delta=|a_dt^d|_\Delta$ and $a_dt^d$ is the dominating monomial.

(d)$\implies$(e) If $a\in R^\times_\eta$ then $\pi^n\in aR$ for some $n$, and hence there exists $a'\in A$ and $u\in A\cap R^\times$ such that $aa'=u\pi^n$. By (i) $|u|_\Delta$ is the constant function 1, hence $|a|_\Delta|a'|_\Delta$ is the constant function $|\pi^n|$, and by the above paragraph, there exists a dominating monomial $a_dt^d$ with $d\in\bfZ^m\times\{0\}$.

(e)$\implies$(f) If (e) holds then by Lemma~\ref{newlem} any monomial of $a$ is divisible by $a_dt^d$, and hence $u=a/(a_dt^d)\in A$. Since $1$ is the free monomial of $u$, the latter is a unit in $R$ by (i).

(f)$\implies$(d) The elements $a_d,t_1\..t_m$ divide some $\pi^n$ and hence $ua_dt^d\in R^\times_\eta$.
\end{proof}

\begin{cor}\label{unitcor}
Let $R$ be as above. An element of $R$ lies in $(R_\eta)^\times$ if and only if it is of the form $a=u\pi'\prod_{i=0}^mt_i^{n_i}$, where $0\neq\pi'\in\kcirc$, $u\in R^\times$ and $n_i\in\bfN$.
\end{cor}
\begin{proof}
The inverse implication is clear. Conversely, if $a\in R\cap(R_\eta)^\times$ then multiplying it by an appropriate $u\in R^\times$ we can assume that $a\in A$ and it remains to use the equivalence of (d) and (f) in Lemma~\ref{locallem}(ii).
\end{proof}

\subsection{Towers of length two}
Now, our goal is to prove that $\eta$-smooth admissible $S$-schemes possess semistable $\tau$-coverings. We start with some particular cases that will be used in the proof. First, let us work out the toric case.

\begin{lem}\label{composlem}
Assume that $k$ is algebraically closed and let $$B=\kcirc[t_0\..t_l]/(t_0\dots t_m-\pi_0),\ A=B[x_0,x_1]/(x_0x_1-\pi_1t_0^{n_0}\dots t_m^{n_m}),$$ where $\pi_0$ and $\pi_1$ are non-zero elements of $\kcirc$. Then $X=\Spec(A)$ possesses a $\tau$-covering $X'$ which is strictly semistable over $S$.
\end{lem}
\begin{proof}
Consider the $(l+1)$-dimensional lattice $M=x_1^\bfZ\times t_1^\bfZ\times\dots\times t_l^\bfZ$. The embedding $A\into\kcirc[M]$ induces a torus embedding $\bfT=\Spec(\kcirc[M])\into X$ which naturally extends to an action of $\bfT$ on $X$, making $X$ a toric $\kcirc$-variety in the sense of \cite{Gubler-Soto}. Using \cite[Proposition~3.3]{Gubler-Soto} the problem can be now translated to a question in convex geometry about covering a polyhedral cone by simplicial ones subject to certain rationality conditions. However, we prefer to deduce the lemma from the classical semistable reduction in the discretely-valued case. If $k$ is of mixed characteristic $(0,p)$ set $r=|p|$, and choose any $r\in(0,1)\cap|k|$ otherwise.

Step 1. {\it The lemma holds when $r_0=|\pi_0|\in r^\bfQ$ and $r_1=|\pi_1|\in r^\bfQ$.} Find a discretely-valued subfield $k_0\subset k$ with an element $\pi\in k_0$ such that $r_0=|\pi|^{n_0}$ and $r_1=|\pi|^{n_1}$ for $n_0,n_1\in\bfN$. We can achieve that $\pi_i=\pi^{n_i}$ simply by multiplying $t_0$ and $x_0$ by units. Then $X$ is the base change of a toric variety $X_0$ over $k_0^\circ$. The main results of \cite{KKMS} imply that after a finite ground field extension, any toric $\kcirc$-variety possesses a strictly semistable modification: use the combinatorial description of toric $\kcirc$-varieties in \cite[IV.3.I]{KKMS} and the subdivision theorem \cite[III.4.1]{KKMS}. Thus, we can choose a subfield $k_0\subseteq k_1\subset k$ such that $k_1/k_0$ is finite and $X_1=X_0\otimes_{k^\circ_0}k^\circ_1$ possesses a modification $X'_1$ which is strictly semistable over $k^\circ_1$. Then $X=X_1\otimes_{k^\circ_1}\kcirc$, and pulling $X'_1\to X_1$ back to $\kcirc$ we obtain a modification $X'\to X$ which is strictly semistable over $S$.

Step 2. {\it The lemma holds when $r_0\in r^\bfQ$.} Choose $\lam\in\kcirc$ such that $$s=|\lam|\in r_1r^\bfQ\cap\left(r_1^{1/2},1\right).$$ For $i=0,1$ consider the $A$-subalgebra $A_i=A[x'_i]$ of $\Frac(A)$, where $x'_i=\frac{x_i}{\lam}$. Since $s>r_1$ we have that $$A_0=B\left[x'_0,x_1\right]/\left(x'_0x_1-\frac{\pi_1}{\lam}t_0^{n_0}\dots t_m^{n_m}\right),A_1=B\left[x_0,x'_1\right]/\left(x_0x'_1-\frac{\pi_1}{\lam}t_0^{n_0}\dots t_m^{n_m}\right).$$ By our choice, $|\frac{\pi_1}{\lam}|=\frac{r_1}{s}\in r^\bfQ$, hence applying Step 1 to $X_i=\Spec(A_i)\to\Spec(B)$ we obtain that $X_i$ possesses a strictly $S$-semistable $\tau$-covering $X'_i$.

We claim that $X_0\coprod X_1\to X$ is a $\tau$-covering and hence $X'_0\coprod X'_1$ is a $\tau$-covering of $X$, thereby proving the step. First, both morphisms $X_i\to X$ are $\eta$-isomorphisms because $x_i$ and $x'_i$ differ by a generic unit. As for the lifting property, we will even show that any admissible valuation $\nu\:\Spec(R)\to\Spec(A)$ lifts to some $\Spec(A_i)$. Indeed, $\nu$ is admissible, hence the homomorphism $\phi\:A\to R$ satisfies $\phi(\lam)\neq 0$ and therefore $\phi$ extends to (or factors through) $A_i$ if and only if $\phi(x_i)\in\phi(\lam)R$. If this is not so for $i=0,1$ then using that $R$ is a valuation ring we obtain that $\phi(x_0x_1)\notin\phi(\lam^2)R$. The latter is impossible since $x_0x_1=\pi_1t_0^{n_0}\dots t_m^{n_m}$ and $\pi_1\in\lam^2 \kcirc$ since $s^2>r_1$.

Step 3. {\it The general case.} This time we will change $Y=\Spec(B)$ in the same way as we have changed $X$ in the proof of step 2. So, choose
$\lam\in\kcirc$ such that $$s=|\lam|\in r_0r^\bfQ\cap\left(r_0^{1/(m+1)},1\right).$$ For each $0\le i\le m$ set $B_i=B[t'_i]$, where $t'_i=\frac{t_i}{\lam}$. Then $$B_i=\kcirc\left[t_0\..t_{i-1},t'_i,t_{i+1}\..t_l\right]/\left(t_0\dots t_{i-1}t'_it_{i+1}\dots t_m-\frac{\pi_0}{\lam}\right).$$ Precisely the same argument as in step 2 shows that each morphism $Y_i=\Spec(B_i)\to Y$ is an $\eta$-isomorphism and the morphism $\coprod_{i=0}^m Y_i\to Y$ is a $\tau$-covering (e.g., if an admissible valuation $\phi\:B\to R$ does not extends to any $B_i$ then $\phi(t_0\dots t_m)\notin\phi(\lam^{m+1})R$, which is impossible since $t_0\dots t_m=\pi_0$ and $\pi_0\in\lam^{m+1}\kcirc$).

Note that the schemes $X_i=Y_i\times_YX$ are of the form $X_i=\Spec(A_i)$ for
$$A_i=B_i\left[x,y\right]/\left(xy-(\pi_1\lam^{n_i})t_0^{n_0}\dots t_{i-1}^{n_{i-1}}t'^{n_i}_it_{i+1}^{n_{i+1}}\dots t_m^{n_m}\right).$$ Since $|\frac{\pi_0}{\lam}|=\frac{r_0}{s}\in r^\bfQ$, each $X_i$ has a strictly $S$-semistable $\tau$-covering $X'_i$ by step 2. Since $X_i$ are admissible, $X_i=(Y_i\times_YX)^\st$ and by compatibility of $\tau$-coverings with strict base changes $X_i$ form a $\tau$-covering of $X$. Hence $X'_i$ form a strictly $S$-semistable $\tau$-covering of $X$.
\end{proof}

The following corollary will play a central role in running induction on the dimension.

\begin{cor}\label{composcor}
Assume that $k$ is algebraically closed, $Y$ is a semistable $S$-scheme, and $f\:X\to Y$ is a nodal curve such that $f_\eta$ is smooth. Then $X$ possesses a $\tau$-covering $X'$ which is strictly semistable over $S$.
\end{cor}
\begin{proof}
The question is \'etale-local on $X$ hence we can work \'etale-locally at a closed point $x\in X$ and its image $y=f(x)$. In particular, we can assume that $Y$ is strictly semistable. In addition, if $f$ is smooth at $x$ then $X$ is semistable at $x$, hence we can assume that $x$ is a nodal point in the fiber over $y$, in particular, $x\in X_s$.

Step 1. {\em Reduction to the case when $Y=\Spec(B)$ is a model semistable $S$-scheme, $y=O$ is the origin, $f$ factors through an \'etale morphism $g\:X\to X_0=\Spec(A)$ with $A=B[v,w]/(vw-b)$, and $x_0=g(x)$ is the singular point of the fiber over $O$.} By Lemma~\ref{loclem} there exists an \'etale morphism $g\:U\to Z$ such that $U\subseteq Y$ is an open neighborhood of $y$, $Z=\Spec(B)$ for $B=\kcirc[t_0\..t_l]/(t_0\dots t_m-\pi_0)$ and $0\neq\pi_0\in\kcirc$, and $g(y)$ is the origin $O$ of $Z$. Since $X\to Z$ is a nodal curve too, replacing $X$ and $Y$ by $f^{-1}(U)$ and $Z$, we can assume that $Y=\Spec(B)$ and $y=O$. Then replacing $X$ by an \'etale neighborhood of $x$ we can also achieve that $X$ is \'etale over $X_0=\Spec(B[v,w]/(vw-b)$. Since $f$ is singular at $x$, the morphism $h\:X_0\to Y$ is singular at the image $x_0$ of $x$.

Step 2. {\em $b\in R_\eta^\times$, where $R=\calO_{Y,O}$.} Let $T\subset X_0$ be the singular locus of $h$; it is mapped isomorphically onto the vanishing locus $V(b)\subset Y$ of $b$. Since $X\to X_0$ is an \'etale neighborhood of $x_0$ and $f$ is $\eta$-smooth, the morphism $h$ is $\eta$-smooth in a neighborhood of $x_0$ and hence $x_0$ is not in the closure of $T_\eta$. Thus, $V(b)_\eta\cap\Spec(R)=\emptyset$ and hence $b\in R_\eta^\times$.

Step 3. {\em End of proof.} By Corollary~\ref{unitcor}, $b=u\pi_1\prod_{i=0}^m t_i^{n_i}$ with $u\in R^\times$ and $0\neq\pi_1\in\kcirc$. Shrinking $X$ we can assume that $f(X)\subset V$, where $V$ is a neighborhood of $O$ such that $u\in\calO^\times(V)$. Set $b'=b/u$ and $X'_0=\Spec(B[v,w]/(vw-b'))$. Since $X_0\times_YV$ is isomorphic to $X'_0\times_YV$, we can replace $X_0$ by $X'_0$ achieving that $b=\pi_1\prod_{i=0}^m t_i^{n_i}$. In this case, $X_0$ possesses a strictly semistable $\tau$-covering $X'_0\to X_0$ by Lemma~\ref{composlem}, and hence $X'=X'_0\times_{X_0}X$ is a required $\tau$-covering of $X$.
\end{proof}

\subsection{Algebraic altered local uniformization}
We start with a simple general lemma.

\begin{lem}\label{covlem}
Let $X$ be a quasi-compact and quasi-separated scheme with a quasi-compact open subscheme $U$ and let $U=\cup_{i=1}^nU_i$ be an open covering such that each $U_i$ is quasi-compact. Then there exists a $U$-modification $X'\to X$ and an open covering $X'=\cup_{i=1}^nX'_i$ such that $X'_i\cap U=U_i$.
\end{lem}
\begin{proof}
By the quasi-compactness assumption, each $Z_i=X\setminus U_i$ is the vanishing locus of a finitely generated ideal $\calI_i\subset\calO_X$. The vanishing locus of $\calI=\sum_i\calI_i$ is $\cap_i Z_i=X\setminus U$, so the blow-up $f\:X'=\Bl_\calI(X)\to X$ is a $U$-modification. By \cite[Lemma~1.4]{con}, $f$ separates the strict transforms $Z'_i$ of the subschemes $Z_i$, hence the open subschemes $X'_i=X'\setminus Z'_i$ form a covering of $X'$ as required.
\end{proof}

Now, we are in a position to prove the main algebraic result of the paper.

\begin{theor}\label{mainth}
Let $k$ be a valued field of height one, $S=\Spec(\kcirc)$ and $\eta=\Spec(k)$. Assume that $X$ is an admissible $\eta$-smooth $S$-scheme. Then there exists a finite extension of valued fields $l/k$ and a covering $X'\to X\times_SS'$ for the $\eta$-\'etale $S$-admissible topology such that $X'$ is strictly semistable over $S'=\Spec(\lcirc)$.
\end{theor}
\begin{proof}
We will freely replace $X$ by $\tau$-coverings throughout the proof.

Step 1. {\em Reduction to the case of an algebraically closed $k$.} Let $K$ be the algebraic closure of $k$ provided with an extension of the valuation, and let $l_i/k$ be the finite subextensions of $K/k$ provided with the induced valuations. Then $S'=\Spec(\Kcirc)$ is the filtered limit of the schemes $S_i=\Spec(l^\circ_i)$. By \cite[$\rm IV_3$, Th\'eor\`eme~8.8.2]{ega}, any finitely presented morphism $h\:X'\to X\times_SS'$ is the base change of a morphism $h_i\:X'_i\to X\times_SS_i$ for a large enough $i$. Furthermore, if $h$ is a $\tau$-covering then already $h_j=h_i\times_{S_i}S_j$ is a $\tau$-covering for a large enough $j$: for $\eta$-\'etaleness this follows from \cite[$\rm IV_4$, Proposition~17.7.8]{ega} and for the lifting property one uses that it holds for $S'\to S_i$ and hence also for $X\times_SS'\to X\times_SS_i$. In the same manner, the two cited results imply that if $X'$ is strictly semistable over $S'$ then already $X'_j=X'_i\times_{S_i}S_j$ is strictly semistable over $S_j$ for a large enough $j$. Thus, it suffices to consider the case when $k=K$.

Step 2. {\em The problem is local on $U=X_\eta$.} Indeed, if $U=\cup_{i=1}^nU_i$ is an open covering, then by Lemma~\ref{covlem} there exists an $\eta$-modification $X'\to X$ and a covering $X'=\cup_{i=1}^nX'_i$ such that $(X'_i)_\eta=X'_i\cap U=U_i$. Since $\coprod_{i=1}^nX'_i\to X$ is a $\tau$-covering, it suffices to prove the theorem for each $X'_i$ separately.

Step 3. {\em We can assume that $X_\eta$ is smooth of pure dimension $d>0$ and there exists a morphism of admissible $S$-schemes $f\:X\to Y$ such that $Y$ and $f$ are $\eta$-smooth and $Y_\eta$ is of pure dimension $d-1$.} We can assume that $X$ is affine, say $X=\Spec(A)$. Since $X_\eta$ is smooth, we can assume that it is of pure dimension $d$ by step 2. Furthermore, if $d=0$ then $X$ is a disjoint union of schemes isomorphic to $S$ and $\eta$ since there are no intermediate subrings $\kcirc\subsetneq A\subsetneq k$. So, we can assume that $d>0$. Using step 2 again, we can assume that there exists a smooth relative curve $g\:X_\eta\to V$ such that $V=\Spec(B)$ is a smooth affine $k$-variety. Let $g^\#\:B\to A_\eta=A\otimes_{\kcirc}k$ be the induced homomorphism. Choose $k$-generators $b_1\..b_n$ of $B$ and non-zero elements $\pi_1\..\pi_n\in\kcirc$ such that $\pi_ig^\#(b_i)\in A$. Consider the $\kcirc$-subalgebra $C=\kcirc[\pi_1b_1\..\pi_nb_n]$ of $B$, then $C_\eta=B$ and $g^\#$ induces a homomorphism $f^\#\:C\to A$. The morphism $f\:X\to Y=\Spec(C)$ is as required since $f_\eta=g$.

We already established the case $d=0$, so assume that the theorem holds for $d-1$ and let us prove that it also holds for $X$. In the sequel, we can also work $\tau$-locally on $Y$: if $Y'\to Y$ is a $\tau$-covering then we can replace $f\:X\to Y$ by the strict base change $(Y'\times_YX)^\st\to Y'$.

Step 4. {\em We can assume, in addition, that $f$ is flat.} Since $f$ is $\eta$-flat we can apply the flattening theorem of Raynaud and Gruson, \cite[Part I, Th\'eor\`em~5.7.9]{RG}, to obtain an $\eta$-modification $Y'\to Y$ such that the strict base change $(Y'\times_YX)^\st\to Y'$ is flat.

Step 5. {\em We can assume, in addition, that $f$ is a nodal curve.} Since $f$ is $\eta$-smooth, the stable modification theorem \cite[Theorem~2.3.3]{temrz} provides a $\tau$-covering $Y'\to Y$ such that the $Y'$-curve $Z=X\times_YY'$ possesses a stable modification $X'$. Moreover, replacing $Y'$ by a larger $\tau$-covering $\tilY$ we can assume that it is a disjoint union of integral schemes: for example, one can simply take $\tilY$ to be the normalization of $Y'$ using the relatively difficult fact that $\tilY$ is of finite type over $S$, see \cite[Theorem~3.5.5]{temst} (a more elementary alternative is to argue straightforwardly similarly to step 2). By \cite[Corollary~1.3(ii)]{temst}, the stable modification morphism $h\:X'\to Z$ is an isomorphism over any open subscheme $U$ of $Z$ which is semistable over $Y'$. Since $Z_\eta$ is such a subscheme (it is even $Y'$-smooth), $h$ is an $\eta$-modification. So, $X'\to X$ is $\eta$-\'etale and hence a $\tau$-covering. Replacing $X$ and $Y$ by $X'$ and $Y'$ we achieve the situation when $X\to Y$ is a nodal curve.

Step 6. {\em End of proof.} By the induction assumption there exists a $\tau$-covering $Y'\to Y$ such that $Y'$ is semistable over $S$. Replacing $Y$ and $X$ by $Y'$ and $X\times_YY'$ we can assume that $Y$ is semistable, and then $X$ possesses a strictly semi-stable $\tau$-covering by Corollary~\ref{composcor}.
\end{proof}

\section{The formal and analytic cases}\label{unifsec}
In this section, $k$ is a non-archimedean analytic field, i.e. a complete real-valued field. In addition, we assume that $k$ is non-trivially valued and use the notation $\gtS=\Spf(\kcirc)$. All completions we consider are with respect to the $(\pi)$-adic topology where $\pi\in k$ satisfies $0<|\pi|<1$.

\subsection{Admissible formal $\kcirc$-schemes}
Recall that a formal $\kcirc$-scheme $\gtX$ is called {\em admissible} if it is flat and of finite type over $\gtS$.

\subsubsection{Generic fiber}
Given an admissible $\gtX$ let $\gtX_\eta^\rig$ denote the associated rigid analytic space over $k$, which is often called the generic fiber of $\gtX$. In fact, the generic fiber construction is functorial and we say that a morphism $f\:\gtY\to\gtX$ between admissible formal $\kcirc$-schemes is {\em rig-\'etale} if its generic fiber $f_\eta^\rig$ is \'etale. In the same manner we define {\em rig-smooth} admissible formal schemes.

\subsubsection{Algebraization}
As earlier, let $S=\Spec(\kcirc)$. The completion functor associates to any admissible $S$-scheme $X$ the admissible formal $\gtS$-scheme $\hatX$. We say that an admissible formal $\gtS$-scheme $\gtX$ is {\em algebraizable} if it is isomorphic to the completion of such an $X$. The following result is an easy consequence of an algebraization theorem of R. \'Elkik.

\begin{theor}\label{algth}
Any affine rig-smooth admissible formal $\kcirc$-scheme $\gtX$ is algebraizable. In addition, one can achieve that $\gtX=\hatX$ where $X$ is an affine $\eta$-smooth admissible $S$-scheme.
\end{theor}
\begin{proof}
Let $\gtX=\Spf(A)$. By \cite[Proposition 3.3.2]{temdes}, $\gtX$ is rig-smooth if and only if the Jacobian ideal $H_{A/\kcirc}$ is open (the cited result is formulated with the assumption that $k$ is discretely-valued, but it is not used in the proof). The latter means that $\gtX$ is formally smooth outside of the closed fiber in the sense of \cite[p.581]{Elk} and hence \cite[Th\'eor\`em~7 on p.582 and Remarque 2(c) on p.588]{Elk} implies that $A=\hatB$ for a finitely generated $\kcirc$-algebra $B$ such that $X=\Spec(B)$ is generically smooth over $S$. Dividing $B$ by the $\kcirc$-torsion we achieve that $B$ becomes $\kcirc$-flat keeping $\hatB$ and $X_\eta$ unchanged.
\end{proof}

\subsection{Rig-\'etale topology}
We provide the category of admissible formal $\kcirc$-schemes with the following {\em rig-\'etale topology}: a morphism $f\:\gtY\to\gtX$ is a rig-\'etale covering if its generic fiber $f_\eta^\rig$ is an \'etale covering.

\begin{rem}
(i) One might be surprised by this definition since its straightforward analogue for $S$-schemes is rather meaningless, e.g. it would define a morphism $X_\eta\to X$ to be a covering. One of the reasons that this works fine for formal schemes is that they are controlled by their generic fiber much better. For example, a morphism $f\:\gtY\to\gtX$ is proper or separated if and only if its generic fiber is so.

(ii) The rig-\'etale topology was introduced in \cite[Definition~3.11]{FRGIII}. We will not need this, but it follows from \cite[Theorem~4.6]{FRGIII} that the rig-\'etale topology is generated by flat rig-\'etale coverings and admissible formal blow-ups. In particular, this indicates that the rig-\'etale topology is the analogue of the topology $\tau$ on the category of $S$-schemes.
\end{rem}

Here is another indication that the rig-\'etale topology is an analogue of $\tau$.

\begin{lem}\label{rigcovlem}
Assume that $X$ is an admissible $S$-scheme and $f\:Y\to X$ is a $\tau$-covering. Then $\hatf\:\hatY\to\hatX$ is a rig-\'etale covering.
\end{lem}
\begin{proof}
First, we claim that $\hatf$ is rig-\'etale. It suffices to consider the affine case, say $f\:\Spec(B)\to\Spec(A)$. We refer to \cite[0.2]{Elk} for the definition of the Jacobian ideal $H_{B/A}$. Since the vanishing locus $V(H_{B/A})$ coincides with the non-smoothness locus of $f$, we have that $\pi^n\in H_{B/A}$ for some $n$. The topological Jacobian ideal $H_{\hatB/\hatA}$ (see \cite[p.507]{temdes}) is defined by the same generators as $H_{B/A}$, hence $H_{\hatB/\hatA}=H_{B/A}\hatB$ contains $\pi^n$ too, and the morphism $\hatY\to\hatX$ is rig-smooth by \cite[Proposition~3.3.2]{temdes}. Counting the dimensions it is easy to see that the generic dimension of $\hatf$ is zero, hence it is, in fact, rig-\'etale.

It remains to show that $\hatf$ is rig-surjective. By Lemma~\ref{coverlem} there exists an $\eta$-modification $Y'\to Y$ such that the composed morphism $Y'\to X$ factors into a composition of a flat covering $Y'\to X'$ and an admissible blow-up $X'\to X$. It suffices to prove that both completions $\hatY'\to\hatX'$ and $\hatX'\to\hatX$ are rig-surjective. Since $\hatX'\to\hatX$ is an admissible formal blow-up, it is even a rig-isomorphism. Concerning the completion of $Y'\to X'$, we can, again, work locally and hence assume that $X'=\Spec(A')$ and $Y'=\Spec(B')$. Since the homomorphisms $A'/\pi^nA'\to B'/\pi^nB'$ are faithfully flat, the completion $\hatA'\to\hatB'$ is faithfully flat by \cite[Lemma~1.6]{FRGI}. In particular, $\hatf$ is rig-surjective.
\end{proof}

\subsection{Formal altered local uniformization}
Recall that a formal $\gtS$-scheme is called {\em strictly semistable} (resp. {\em semistable}) if locally (resp. \'etale-locally) it admits an \'etale morphism to a model formal scheme of the form $$\Spf(\kcirc\{t_0\..t_l\}/(t_1\dots t_m-\pi))$$ with $0\neq\pi\in\kcirc$. For example, if $X$ is (strictly) semistable over $S$ then its completion is (strictly) semistable over $\gtS$.

\begin{theor}\label{formalth}
Assume that $k$ is a complete valued field of height one and $\gtX$ is a rig-smooth admissible formal scheme over $\gtS=\Spf(\kcirc)$. Then there exists a finite extension of valued fields $l/k$ and a rig-\'etale covering $\gtX'\to\gtX\times_\gtS\gtS'$, where $\gtS'=\Spf(\lcirc)$, such that $\gtX'$ is strictly semistable over $\gtS'$.
\end{theor}
\begin{proof}
Replacing $\gtX$ by its affine covering we can assume that it is affine. Then by Theorem~\ref{algth}, $\gtX=\hatX$ for an $\eta$-smooth admissible $S$-scheme $X$, where $S=\Spec(\kcirc)$. By Theorem~\ref{mainth} there exists a finite extension $l/k$ and a $\tau$-covering $X'\to X\times_SS'$ with a strictly semistable $X'$, where $S'=\Spec(\lcirc)$. By Lemma~\ref{rigcovlem}, passing to completions we obtain a rig-\'etale covering $\hatX'\to\hatX\times_\gtS\gtS'$. Since (strict) semistability is preserved by the completion functor, $\gtX'=\hatX'$ is as required.
\end{proof}

As a corollary, we obtain the generalization of Hartl's theorem to arbitrary ground fields.

\begin{cor}\label{Hartl}
Assume that $k$ is a complete real-valued field and $X$ is a smooth quasi-compact and quasi-separated rigid space over $k$. Then there exists a finite extension $l/k$ and an \'etale covering $X'\to X\otimes_kl$ such that $X'$ is affinoid and possesses a strictly semistable affine formal model.
\end{cor}
\begin{proof}
By Raynaud's theorem \cite[Theorem~4.1]{FRGI}, $X$ admits an admissible formal model $\gtX$. Find $l/k$ and $\gtX'\to\gtX\times_\gtS\gtS'$ as in Theorem~\ref{formalth}. Replacing $\gtX'$ by its affine covering we can also achieve that $\gtX'$ is affine. Passing to the generic fibers we obtain an \'etale covering $X'=\gtX'_\eta$ of $X\otimes_kl$ such that $X'$ is affinoid and possesses the strictly semistable model $\gtX'$.
\end{proof}

\subsection{Translation to Berkovich geometry}
Recall that the category of quasi-compact and quasi-separated rigid spaces over $k$ is equivalent to the category of compact strictly $k$-analytic Berkovich spaces. This equivalence matches \'etale and smooth morphisms of rigid spaces with quasi-\'etale and quasi-smooth morphisms of Berkovich spaces (the notions of \'etale and smooth morphisms are reserved in Berkovich geometry for boundaryless ones). Therefore, Corollary~\ref{Hartl} translates to the language of Berkovich spaces as follows:

\begin{theor}\label{berth}
Assume that $X$ is a quasi-smooth compact strictly $k$-analytic space. Then there exists a finite extension of analytic fields $l/k$ and a surjective quasi-\'etale morphism $X'\to X\otimes_kl$ such that $X'=\calM(\calA)$ is affinoid and its maximal affine formal model $\Spf(\calAcirc)$ is semistable over $\lcirc$.
\end{theor}

\bibliographystyle{amsalpha}
\bibliography{weak_uniformization}

\end{document}